\documentclass[reqno]{amsart}

\usepackage[dvipsnames]{xcolor}
\usepackage{todonotes}

\usepackage{latexsym}
\usepackage{amssymb}
\usepackage{amsmath}
\usepackage{amsthm}
\usepackage{amscd}
\usepackage[normalem]{ulem} %

\theoremstyle{plain}
\newtheorem{thm}{Theorem}[section]
\newtheorem*{thm*}{Theorem}
\newtheorem{lem}[thm]{Lemma}
\newtheorem{defin}[thm]{Definition}
\newtheorem{prop}[thm]{Proposition}
\newtheorem{prop*}{Proposition}

\theoremstyle{definition}
\newtheorem{rem}[thm]{Remark}

\newtheorem{ques}[thm]{Question}
\newtheorem{cor}[thm]{Corollary}

\newcommand{\suchthat}{\text{ s.t. }} 

\newcommand{\ssepstatespace}{\Z}
\newcommand{\configspace}{X}

\makeatletter
\def\blx@maxline{77}
\makeatother

\newcommand{\R}{\mathbb{R}} %
\newcommand{\Z}{\mathbb{Z}} %
\newcommand{\Prob}{\mathbb{P}} %

\newcommand{\e}{\epsilon}
\newcommand{\w}{\omega}
\newcommand{\W}{\Omega}
\newcommand{\Norm}[2]{\left\lVert{#1}\right\rVert_{#2}} %
\newcommand{\AND}{\textrm{ and }}

\newcommand{\almostsurely}{\textrm{a.s}}

\renewcommand{\limsup}{\varlimsup}

\newcommand{\E}{\mathbb{E}}

\newcommand{\specialpts}{\mathcal{S}} %
\newcommand{\biinfinite}{bi-infinite}
\newcommand{\Biinfinite}{Bi-infinite}

\newcommand{\directions}{\operatorname{\mathbf{A}}}
\newcommand{\borel}{\mathcal{B}}
\newcommand{\origin}{\mathbf{0}}

\newcommand{\OFP}{\Omega,\mathcal{F},\Prob}

\newcommand{\tmap}{T}
\newcommand{\tmaps}[1]{\{T^z\}_{z \in \Z^{#1}}}

\def\walk{X}
\newcommand{\walks}{\mathcal{W}}
\newcommand{\rect}{\operatorname{Rect}}

\DeclareMathOperator{\argmin}{argmin}
\newcommand{\Uniform}{\operatorname{Uniform}}
\newcommand{\Bernoulli}{\operatorname{Bernoulli}}

\newcommand{\Eqref}[1]{\eqref{#1}}

\newcommand{\Secref}[1]{Section~\ref{#1}}
\newcommand{\Thmref}[1]{Theorem~\ref{#1}}
\newcommand{\Corref}[1]{Corollary~\ref{#1}}
\newcommand{\Lemref}[1]{Lemma~\ref{#1}}
\newcommand{\Propref}[1]{Prop.~\ref{#1}}
\newcommand{\Defref}[1]{Definition~\ref{#1}}

\usepackage[urlcolor=blue,colorlinks=true,linkcolor=blue,citecolor=blue]{hyperref}

\usepackage{versions}
\includeversion{note}       %
\includeversion{arjunnote}  
\includeversion{oldnote}    %
\includeversion{longcalc}   %
\markversion{note}
\markversion{arjunnote}
\markversion{oldnote}
\markversion{longcalc}

\usepackage[utf8]{inputenc}

\excludeversion{openquestion}
\excludeversion{longcalc}
\excludeversion{note}
\excludeversion{arjunnote}
\excludeversion{oldnote}



\usepackage[finalnotes]{optional}

\title{Stationary coalescing walks on the lattice II: entropy}
\author{Jon Chaika}
\address{Department of Mathematics\\ University of Utah\\ Salt Lake City, UT}
\email{chaika@math.utah.edu}
\author{Arjun Krishnan}
\address{Department of Mathematics, University of Rochester, Rochester, NY}
\email{arjun.krishnan@rochester.edu}

\date{}

\usepackage[backend=biber,natbib=true,style=numeric]{biblatex}
\renewcommand{\specialpts}{\W_\alpha}

\begin{document}

\begin{abstract}
    This paper is a sequel to Chaika and Krishnan, 2016. We again consider translation invariant measures on families of nearest-neighbor semi-infinite walks on the integer lattice $\Z^d$. We assume that once walks meet, they coalesce. We consider various entropic properties of these systems. We show that in systems with completely positive entropy, bi-infinite trajectories must carry entropy. In the case of directed walks in dimension $2$ we show that positive entropy guarantees that all trajectories cannot be bi-infinite. To show that our theorems are proper, we construct a stationary discrete-time symmetric exclusion process whose particle trajectories form bi-infinite trajectories carrying entropy.
\end{abstract}
\maketitle

\tableofcontents

\section{Introduction}
\label{sec:intro}
Let $(\W,\mathcal{F},\Prob,\{T^z\}_{z \in \Z^d})$ be a $\Z^d$ measure-preserving dynamical system. Let $\walks(\w)$ be a stationary or translation-covariant subset ($\walks(T^z\w) = \walks(\w) - z$) of the lattice $\Z^d$, and suppose that it has at least one point with positive probability; i.e., $\Prob( \origin \in \walks(\w)) > 0$. Consider a family of measurable walks on the lattice $\{\walk_z\}_{z \in \walks}$, where each $\walk_z \colon \Omega \times \Z^+ \to \Z^d$ is a nearest-neighbor path that starts at $z$. We assume that almost surely for all $k \in \Z^+$ and $z \in \walks(\w)$, these walks have been created in a stationary way: 
\[
    \walk_z(\w, k) = x + \walk_{z-x}(\tmap^x \w, k),
\]
and that they are \emph{compatible}: 
\begin{align*}
    X_z(\w,k+1) = X_{X_z(\w,1)}(\w,k) .
\end{align*}
The compatibility condition implies that if two walks meet at a point at some time, then they remain together in the future; i.e., the two walks must coalesce and \emph{cannot cross each other}. 

Because walks coalesce when they meet, we may assume that there is a stationary vector-field $\alpha$ that is the discrete time-derivative of the walks:
\begin{equation}
    \alpha(\w,z) = \walk_z(\w,1) - \walk_z(\w,0).
    \label{eq:direction function definition}
\end{equation}
The $\alpha$ function takes values in $\directions \subset \{ \pm e_1, \ldots \pm e_d \}$, and we call a particular $\alpha$ value an \emph{arrow}. In this paper, we focus our efforts on \emph{directed} walks, where $\directions = \{e_1,\ldots,e_d\}$. One could think of the walks as the flow generated by the stationary vector field of arrows.  

A \biinfinite{} trajectory is a walk that is infinite in both forward and backward directions (see \Defref{def:biinfinite trajectory definition}). When all walks in $\walks$ coalesce with probability one, we say that we have almost-sure coalescence. In Theorem 2.5 and Theorem 2.9 of \citep{chaika_stationary_walks_2016}, we proved that in $d=2$, assuming that that walks do not form loops and cross every vertical (or horizontal or diagonal) line a finite number of times, there is a \emph{behavioral dichotomy}: Walks coalesce almost surely or there must exist a positive density of \biinfinite{} trajectories that \emph{do not coalesce with each other}. Directed paths (up-and-right paths) in $d=2$ automatically satisfy the no loops and line-crossing conditions. Moreover,
\begin{enumerate}
    \item Bi-infinite trajectories themselves form measure-preserving dynamical systems (with translation along the \biinfinite{} trajectory). Thus, all the walks have the same asymptotic direction (Corollary 2.14  of \citep{chaika_stationary_walks_2016}). 
    \item All trajectories that are not \biinfinite{} must coalesce with \biinfinite{} ones (Corollary 2.7 of \citep{chaika_stationary_walks_2016}). 
\end{enumerate}
In the same paper, we also provide a simple periodic system with \biinfinite{} trajectories and a system of independent identically distributed (iid) arrows with almost-sure coalescence, which seems to suggest that the more ``random'' a system is, the more likely we are to have almost-sure coalescence. Phrased as a question: 
\begin{quote}
    In $d=2$, is there a natural notion of randomness {(mixing, entropy, etc.)} that distinguishes between the almost-surely coalescent and the \biinfinite{} trajectories cases?
\end{quote}
Positive entropy is a useful dynamical measure of randomness. If we have directed walks in $d=2$, we show that positive entropy is enough to ensure that \emph{not all} trajectories can be \biinfinite{} (\Thmref{thm:bi-infinite points cannot have full measure with positive entropy}). This motivated us to investigate the general implications positive entropy on \biinfinite{} trajectories in this paper.  

Our abstract model of coalescing walks is motivated by questions about the behavior of infinite geodesics in first- and last-passage percolation. Let $\W = \{\w_z \in \R\}_{z \in \Z^d}$ with product $\sigma$-algebra and a translation invariant measure $\Prob$. The $\w_z$ are called \emph{weights} and they are typically nonnegative random variables. Let $\walk_{x,y}$ be a path from $x$ to $y$ and let the total weight of the path be the sum $W(\walk_{x,y}) := \sum_{z \in \walk_{x,y}} \w_z$. Define the first-passage time from $x$ to $y$ to be
\[
    T(x,y) = \inf_{\walk_{x,y}} W(\walk_{x,y}).
\]
If the weights are strictly positive, the first-passage time $T(x,y)$ defines a random metric on the lattice $\Z^d$. A geodesic for this random metric is a nearest-neighbor path that minimizes the passage time between every point that lies on it.

Several groups have created stationary compatible families of semi-infinite geodesics. The first to do so were \citet{MR1387641} under reasonable but unproven hypotheses on the time-constant of first-passage percolation. Importantly, they show that if the weights have the finite-energy property (see \citep[Hypothesis C, Prop. 9]{MR648202} or \citep[Definition 1.9]{MR3152744}), \biinfinite{} trajectories do not exist in these families. Other notable constructions include \citet{MR3152744} in first-passage percolation, \citet{MR3704768} in last-passage percolation, and more recently, by \citet{2016arXiv160902447A} under weaker assumptions than \citep{MR3152744} in first-passage percolation. In all of these cases, finite-energy and the Licea-Newman argument \citep{MR1387641} is used to show that \biinfinite{} trajectories do not exist in stationary compatible families of geodesics. In other words, finite-energy is an obstruction to the existence of \biinfinite{} geodesics. 

The finite-energy assumption is quite reasonable in the first- and last-passage percolation context since systems with iid weights do possess the finite-energy property. However, it does not generalize nicely to our more general setting of walks here. One appealing replacement candidate for finite-energy is completely positive entropy, since this property is inherited by all factors of the original system. In \Thmref{thm:special trajectories have positive entropy}, we show that in systems with completely positive entropy, \biinfinite{} trajectories must carry some part of the entropy of the system. We provide an example in \Secref{sec:symmetric simple exclusion} where the \biinfinite{} trajectories do indeed have positive entropy, showing that simply positive entropy by itself is not an obstruction to \biinfinite{} trajectories. The example is based on a discrete-time symmetric simple exclusion process (SSEP), and the particle trajectories form \biinfinite{} trajectories here. We show that these particle trajectories must have positive entropy.

\subsection{Acknowledgments}
A.\, Krishnan thanks Eric Cator for suggesting the exclusion example. J.\ Chaika was supported in part by NSF grants DMS-135500 and DMS-1452762, the Sloan foundation and a Warnock chair.
A.\ Krishnan was supported in part by a Simons collaboration grant 638966.

\section{Main results}
\label{sec:main results}
Let  $(\OFP)$ be a probability space and $\{T^g\}_{g \in G}$ be a measure-preserving group action of a countable group $G$ on $\W$. Since the $\sigma$-algebra does not feature prominently in any of our proofs, we will usually omit it from the notation. The corresponding dynamical system will be written as $(\W,\Prob,G)$, $(\W,\Prob,\{T^x\}_{x \in G})$ or as $(\W,\Prob,T)$ if the group is $\Z$.

\begin{defin}[Rectangular subsets of $\Z^d$]
    Let the rectangle centered at $x \in \Z^d$ with side lengths $(N_1,\ldots,N_d)$ be
    \[
        \rect_x(N_1,\ldots,N_d) = \prod_{i=1}^d [x_i - N_i, x_i + N_i] .
    \]
\end{defin}
If the side-lengths are equal, then we write $\rect_x(N)$. The boundary of any $R \subset \Z^d$ is written as $\partial R$ and consists of the set of points in $R$ that have at least one point in $\Z^d \setminus R$ as a nearest neighbor.

We call an arrow configuration \emph{non-trivial} if $\alpha$ is not constant almost surely. The canonical walk $\walk(\w)$ starts at the origin and $\alpha(\w)$ is its (discrete) derivative at time $0$. We will omit the $\w$ from the notation when it is clear from context. We frequently speak of \emph{configurations} on the lattice: for any $\w$, this refers to the collection of walks $\{\walk_z(\w)\}_{z \in \walks}$ or equivalently, the collection of arrows $\{ \alpha(T^z \w) \}_{z \in \walks}$.

\begin{defin}[Coalescence of points]
    Given a configuration, we say that the points $x$ and $y$ coalesce if the walks $\walk_x$ and $\walk_y$ coalesce in the future. That is, for some $k_0,k_1 \in \Z^+$, $X_x(\w,k_0) = X_y(\w,k_1)$. We say we have \emph{almost sure coalescence} if almost surely for all $x,y \in \Z^d$, the walks through $x$ and $y$ coalesce. %
\end{defin}

\begin{defin}[\Biinfinite{} walks and points]
    We say that a point $z \in \Z^d$ is \biinfinite{} if there is a sequence of points $\{a_n\}_{n=0}^\infty \in \Z^d$ such that for each $n$, $\walk_{a_n}(\w,i) = a_{n-i}$ for $i=1,\ldots,n-1$ and $\walk_{a_n}(\w,n) = z$. We call this union of (one-sided) walks $\cup_{n \in \Z^+} \cup_{i=0}^{\infty} \walk_{a_n}(\w,i)$ a \biinfinite{} trajectory. 
    \label{def:biinfinite trajectory definition}
\end{defin}

This paper investigates the extent to which randomness in the scenery influences, excludes, and does not exclude bi-infinite random walks. The most studied form of randomness in the subject, finite energy, has been shown to almost surely exclude bi-infinite random walks under various additional (and often strong) assumptions \citep{newman_surface_1995,MR1387641,MR648202,}. We focus instead on entropy, and in particular, \emph{completely positive entropy}, a property that is intrinsic to the measure preserving system (that is, invariant under isomorphism) and is even preserved by factors. As topologically mixing Markov chains have completely positive entropy, completely positive entropy does not imply finite energy.  On the other hand, C. Hoffman communicated to us that finite energy does not imply completely positive entropy, so there are no implications between these two settings. For all of the results in this paper, we restrict to directed paths in $d=2$. 

For any measure $\nu$ supported on a finite alphabet $A = \{a_1,\ldots,a_n\}$, the Shannon entropy is defined as $H(\nu) = \sum_{i=1}^n \nu(a_i) \log \nu(a_i)$. Let $\W := 
\directions^{\Z^d}$ be the space of arrow configurations. 
Let $M = \rect_0(N_1,\ldots,N_d)$ be a rectangle, let $\pi_M \colon \directions^{\Z^d} \to \directions^{M}$ be the coordinate projection map, and let $\Prob \circ \pi_M^{-1}$ be the pushforward measure.
The entropy-rate of $(\directions^{\Z^d},\Prob,\Z^d)$ is defined as usual by
\begin{equation}
    h(\Prob) = \lim_{N_1,\ldots,N_d \to \infty} |M|^{-1} H(\Prob\circ \pi_M^{-1}).
    \label{eq:entropy rate as a limit of entropy in rectangles}
\end{equation}
Note that the rate at which the $N_j \to \infty$ does not matter \citep[Theorem 15.12]{MR2807681}.

A factor of a dynamical system $(\Omega,\Prob,\tmaps{d})$ is another space $(Y,\Prob',\{S^z\}_{z \in \Z^d})$ such that there is a measurable map $\phi \colon \W \to Y$ with the following properties: $\Prob'$ is simply the pushforward measure $\Prob \circ \phi^{-1}$, and for all $z \in \Z^d$, $\phi(\tmap^z \w) = S^z \phi(\w)$ almost surely. If $Y = A^{\Z^d}$ where $A$ is a finite alphabet, then the factor is called a $\Z^d$ shift-system with finite alphabet.   

In light of \Thmref{thm:bi-infinite points cannot have full measure with positive entropy}, it is natural to consider the \emph{completely positive entropy} condition as a natural alternative to finite-energy. A system has completely positive entropy if all of its factors have positive entropy. In particular, when $(\W,\Prob,\Z^d)$ is a product space with product measure, then it and all of its factors have completely positive entropy \citep{MR0274717}. Completely positive entropy is also equivalent to the fact that every non-trivial partition of $\W$ has positive entropy. To explain this statement, we define the usual Kolmogorov-Sinai entropy. A partition $\mathcal{P} = \{P_1,\ldots,P_k\}$ breaks up $\Omega$ into a finite number of pairwise disjoint sets. Given two finite partitions $\mathcal{P}$ and $\mathcal{Q}$ of $\Omega$, define the join or common refinement of the partition as
\[
    \mathcal{P} \lor \mathcal{Q} = \{ A \cap B \colon A \in \mathcal{P}, B \in \mathcal{Q} \}.
\]
Indeed, partitions define factors in the following manner. Let $\mathcal{P}$ be a partition, and let $P(\w) \in \mathcal{P}$ be the partition element that $\w$ belongs to, and assign to each point $\w$ its partition ``address'': let $\phi(\w) = \{P(\tmap^z \w)\}_{z \in \Z^d}$. Then, $(\mathcal{P}^{\Z^d},\Prob \circ \phi^{-1})$ is a factor under the $\Z^d$ shift. 
Let the entropy of the partition $\mathcal{P}$ be 
\[
    H(\Prob , \mathcal{P})=\sum_{A\in \mathcal{P}} -\Prob(A)\log(\Prob(A)). 
\]
The entropy rate of the partition is defined as the entropy rate of the shift-system factor $\phi \colon \W \to \mathcal{P}^{\Z^d}$, and is denoted $h(\Prob,\mathcal{P})$ Thus, we also have
\begin{equation}
    h(\Prob,\mathcal{P}) = \lim_{L \to \infty} \frac{1}{|\rect(L)|} H\left(\Prob,\bigvee_{z \in \rect(L)}T^z \mathcal{P} \right).
    \label{eq:entropy rate in terms of join}
\end{equation}

The entropy rate $h(\Prob)$ is defined as the supremum over partitions:
\begin{equation}
    h(\Prob) = \sup_{\mathcal{P}} h(\Prob , \mathcal{P}).
    \label{eq:entropy definition in terms of partitions}
\end{equation}
This is the same entropy rate that appears in \eqref{eq:entropy rate as a limit of entropy in rectangles}. When it is more convenient to speak in terms of partitions and there is no ambiguity about the measure, we will write $h(\mathcal{P})$ and $H(\mathcal{P})$. A generating partition is one that generates the $\sigma$-algebra, and it is a standard fact that it also achieves the supremum in \eqref{eq:entropy definition in terms of partitions}. 

Under the assumption of completely positive entropy, we could not show that the \biinfinite{} trajectories cannot occur, but instead we show that they must carry some of the entropy of the system. We explain what we mean by this last statement next. The arrows induce a map $T_{\alpha}$ along walks defined by
\begin{equation}
    T_{\alpha}\w = T^{\alpha(\w)}\w  .
\label{eq:translation map along walks}
\end{equation}
The $T_{\alpha}$ map is neither measure preserving nor invertible in general. Along \biinfinite{} trajectories, however, it is both invertible and measure preserving. 
Let $\specialpts$ be the event that the origin is in a \biinfinite{} trajectory. For any measurable $A$, define $\Prob_{\alpha}(A) = \Prob(A  \cap \specialpts)$ and obtain the measure space $(\W_{\alpha},\Prob_{\alpha})$ in the usual way. Theorem 2.13 and Corollary 2.14 of \citep{chaika_stationary_walks_2016} state that $(\W_{\alpha},\Prob_{\alpha},T_{\alpha})$ form a measure-preserving $\Z$ system.

In the following theorem, we assume without loss of generality that $\W = \directions^{\Z^d}$ since we may always descend to this factor. 
\begin{thm}
    In $d \geq 2$, let the paths be directed $(\directions = \{e_1,\ldots,e_d\})$ and let $(\directions^{\Z^d},\Prob,\{T^z\}_{z\in \mathbb{Z}_d})$ have completely positive entropy. Let $(\specialpts,\Prob_{\alpha},T_{\alpha})$ be the $\Z$-system on \biinfinite{} trajectories, and suppose $\Prob(\specialpts) > 0$. Then, the entropy-rate of the \biinfinite{} trajectory system is positive almost surely; i.e., 
    \[
        h(\theta_{\w}) > 0 \quad \almostsurely~\w,
    \]
    for each $\theta_{\w}$ an ergodic component of $\Prob_{\alpha}$. 
    \label{thm:special trajectories have positive entropy}
\end{thm}
For the completeness of our treatment, the next theorem shows that the qualitative result proven in the previous theorem is insufficient to rule out bi-infinite random walks:
\begin{thm} There exists a system $(\directions^{\Z^2},\Prob,\{T^z\}_{z\in \mathbb{Z}_2})$ that almost surely has bi-infinite random walks and so that \biinfinite{} trajectory system has positive entropy. 
\label{thm:positive entropy trajectories}
\end{thm}
The system $(\directions^{\Z^2},\Prob,\{T^z\}_{z\in \mathbb{Z}_2})$ is a factor of a discrete-time simple exclusion process we describe in \Secref{sec:symmetric simple exclusion}. 

\section{Proofs}
    
\subsection{Positive entropy}
As a warm up to proving \Thmref{thm:special trajectories have positive entropy} we first show that in any system with positive entropy that has directed walks from every point, not all trajectories can be \biinfinite{}.
\begin{thm}
    In $d=2$, suppose we have directed walks $(\directions = \{e_1,e_2\})$ defined on the $\Z^2$ shift space $(\directions^{\Z^2},\Prob,\Z^2)$. Then, if the entropy-rate $h(\Prob)$ is positive, all trajectories cannot be \biinfinite{}.
   \label{thm:bi-infinite points cannot have full measure with positive entropy}
\end{thm}

In this section, let $\W = \directions^{\Z^d}$ be the space of arrow configurations and let $\borel$ be the product Borel $\sigma$-algebra on it. If a finite-alphabet $\Z^d$ system has positive entropy, the Shannon-MacMillan theorem applies. We first state a corollary of the general Shannon-MacMillan theorem that we state later (see \Thmref{thm:shannon macmillan non ergodic case}). 

\begin{cor}[Shannon-McMillan for ergodic measures]
    Let $(\directions^{\Z^d},\Prob,\Z^d)$ be a measure-preserving ergodic $\Z^d$ system with entropy-rate $h$. For any $\e > 0$, there is a large $L$ such that whenever $R \subset \Z^2$ is a rectangle of minimal side-length $L$, $\exists \mathcal{Y} \subset \directions^R$ such that  $\Prob(\mathcal{Y}) > 1 - \e$, and for every $a \in \mathcal{Y}$,
            \[
                e^{-(h+\e)L^d} < \Prob( a ) < e^{-(h-\e)L^d}.
            \]
    Here, $\Prob(a) = \Prob(\pi_R^{-1}(a))$ is the pushforward measure under the coordinate projection map $\pi_R \colon \directions^{\Z^d} \to \directions^R$.
    \label{cor:shannon macmillan for ergodic measures}
\end{cor}
This implies, in particular, {that $|\mathcal{Y}|$ grows exponentially as a function of $L^d$} since
\begin{equation}
    |\mathcal{Y}| e^{-(h-\e)L^d} \geq \sum_{a \in \mathcal{Y}} \Prob( a) \geq (1-\e).
    \label{eq:shannon mcmillan number of configurations}
\end{equation}
\begin{proof}[Proof of~\Thmref{thm:bi-infinite points cannot have full measure with positive entropy}]
    Suppose for the sake of contradiction that the set of bi-infinite points has full measure; i.e., $\Prob(\specialpts) = 1$. Then almost surely, there is a bi-infinite trajectory through every point on $\Z^2$. By Theorem 2.9 of \citep{chaika_stationary_walks_2016}, no two bi-infinite trajectories can coalesce.

    For $i \in \Z$ and $L$ an even integer, let $R_i = \cup_{k=-L/2}^{L/2} \{i u + k v \}$ and let $R = \cup_{i=0}^{L-1} R_i$, be a rectangle aligned with the vectors $u = e_1 + e_2$ and $v = -e_1 + e_2$. We will count the number of configurations in $\pi_R(\directions^{\Z^2})$ and use Corollary~\ref{cor:shannon macmillan for ergodic measures} to produce a contradiction. Let $x \in R_0$, a point on the southwest boundary of $R$. There are two possibilities for the next step of the walk $\walk_x$: $x+e_1$ or $x+e_2$. Suppose first that $\alpha(x) = e_1$. The point $x+e_2$ must have an ancestor in $R_0$, for if not, there is no \biinfinite{} trajectory passing through it and this contradicts our assumption. Therefore, $x+v \in R_0$ must be the ancestor of $x+e_2$; i.e., $\alpha(x+v) = e_1$. Similarly, we must have $\alpha(x - v) = e_1$. Otherwise, the \biinfinite{} trajectories from $x$ and $x-v$ would coalesce; \biinfinite{} trajectories may not coalesce. 

    Fixing $\alpha$ on any single point $x \in R_0$ determines $\alpha(x - v)$ and $\alpha(x + v)$, and thus, on all of $R_0$. So a single trajectory of length $L$ starting from $x \in R_0$ determines $\alpha$ at one point on each $R_i$. Since this trajectory must be a part of a \biinfinite{} trajectory, by the previous argument, it determines $\alpha$ on all points in $R_i$, $i=0,\ldots,L-1$. Since $\directions = \{e_1,e_2\}$, there are most $2^L$ different trajectories {of length $L$} starting from a single point, and therefore, the total number of allowed configurations in $R$ must satisfy $|\pi_R(\directions^{\Z^2})| \leq 2^L$ almost surely. Since $h(\Prob) > 0$, this contradicts~\eqref{eq:shannon mcmillan number of configurations}, which implies that $|\pi_R(\directions^{\Z^2})| \geq \frac{1}{2} e^{(h(\Prob)/2) L^2}$ for all large enough $L$. 
\end{proof}

Next, we prove~\Thmref{thm:special trajectories have positive entropy}, which states that the \biinfinite{} trajectories must carry some of the entropy of the system when the system has completely positive entropy. Let $\mathcal{I}$ be the invariant $\sigma$-algebra of $T_{\alpha}$. Since $(\W_{\alpha},\borel,\Prob_{\alpha})$ is a separable metric space, it has a regular conditional probability given $\mathcal{I}$ (called $\theta_\w$). Hence it has the following ergodic decomposition:
\[
    \Prob_{\alpha} = \int \theta_{\w} \Prob_{\alpha}(d\w).
\]
Since this system may or may not be ergodic, Theorem~\ref{thm:special trajectories have positive entropy} needs the generalized version of the Shannon-Macmillan theorem.
The generalized Shannon-MacMillan theorem relates the entropy rates of individual ergodic components to the information function. This is a straightforward consequence of the usual Shannon-Macmillan theorem and the ergodic decomposition. 
\begin{thm}\label{thm:SMgen}[generalized Shannon-MacMillan \citep{MR725559}] 
    Let $(\directions^{\Z^d},\Prob,\Z^d)$ be a measure-preserving $\Z^d$ system and $R \subset \Z^d$ be a finite rectangle.  Let the information function $f_R \colon \directions^{\Z^d} \to \R$ be
    \[
        f_R(\w) = 
        \begin{cases}
            - \log\Prob(\pi_R^{-1}(\pi_R(\w)))    &   \Prob(\pi_R^{-1}(\pi_R(\w))) > 0\\
            0   & \text{otherwise}
        \end{cases}.
    \]
    Then,
    \[
        \lim_{|R| \to \infty} \Norm{ |R|^{-1} f_R(\w) - h(\theta_{\w}) }{1} = 0,
    \]
  where $\theta_\w$ is the conditional probability measure given the invariant $\sigma$-algebra, and the notation $|R| \to \infty$ means that the length of the smallest side of the $R$ goes to infinity. 
\label{thm:shannon macmillan non ergodic case}
\end{thm}
It follows from the Markov inequality that for any $\e > 0$, there is an $L$ large enough such that for all rectangles $R$ with minimal side-length larger than $L$,  $\exists \mathcal{Y} \subset \directions^{\Z^d}$ with $\Prob(\mathcal{Y}) \geq 1 - \e$, and for all $\w \in \mathcal{Y}$, we have,
\begin{equation}
    e^{-(h(\theta_{\w}) + \e)|R|} \leq  \Prob(\pi_R^{-1} \circ \pi_R(\w)) \leq e^{-(h(\theta_{\w}) - \e)|R|}.
    \label{eq:corollary of the nonergodic shannon macmillan theorem}
\end{equation}
If $\Prob$ is ergodic, \Corref{cor:shannon macmillan for ergodic measures} follows from~\Eqref{eq:corollary of the nonergodic shannon macmillan theorem}.

The idea behind the proof of \Thmref{thm:special trajectories have positive entropy} is to assume for the sake of contradiction that there exists a subset $\mathcal{C} \subset \specialpts$ of \biinfinite{} trajectories that have zero entropy. We then ``recode'' the arrow configurations so that the zero-entropy bi-infinite trajectories completely determine the arrow configurations on the entire lattice. This new recoded system is a factor of the original system. Due to the completely positive entropy assumption, this factor must also have positive entropy. However, since all of the entropy in the system is concentrated on the bi-infinite trajectories in the set $\mathcal{C}$, they could not have had zero entropy to start with.

{Fix any non-trivial measurable $\mathcal{C} \subset \directions^{\Z^d}$; let $\hat{C}(\w)$ be the set of points $z \in \Z^d$ closest in $\ell^\infty$ distance from the origin such that $T^z \w \in \mathcal{C}$. We define a function $\ell \colon \directions^{\Z^d} \to \Z^d$ that maps the origin to a well-defined point in $\hat{C}(\w)$. In the recoding lemma that follows, we use this function $\ell$ to determine the value of the arrow at $x$ in the recoded configuration by copying the value of the arrow at $\ell(T^x \w)\in \hat{C}(T^x\w)$. In the lexicographic ordering of $\Z^d$, $x < y$ if for some $i \in \{1,\ldots,d\}$, $x_j = y_j$, $j < i$ and $x_i < y_i$. For any finite subset $B \subset \Z^d$, let $m(B)$ be the smallest element in the lexicographic ordering of $B$.} Let 
    \begin{equation}  
        \ell(\w) = m(\hat{C}(\w)).
        \label{eq:definition of l function mapping the origin to the nearest point in C}
    \end{equation}
The $\ell(\w)$ function is almost surely well-defined and measurable, and hence by construction we have the following lemma.
\begin{lem}[Recoding lemma]
    Given any non-trivial $\mathcal{C} \subset \specialpts$ that is invariant under $T_{\alpha}$ and the corresponding $\ell(\w)$ defined above, $\phi \colon \directions^{\Z^d} \to \directions^{\Z^d}$ defined by
    \[
        \phi(\w) = \{\alpha(T^z\w,\ell(T^z\w))\}_{ z \in \Z^d}
    \]
    is a well-defined factor map on $(\directions^{\Z^d},\borel,\Prob)$. 
    \label{lem:recoding lemma}
\end{lem}
Let $\W' := \phi(\directions^{\Z^d})$. The recoding procedure has the following three properties that we state without proof.
\begin{enumerate}
    \item $\W'$ has completely positive entropy.
    \item Let $C(\w) = \{x \in \Z^d \colon \tmap^x\w  \in \mathcal{C} \}$. Suppose $\w_1, \w_2 \in \directions^{\Z^d}$ such that $C(\w_1) = C(\w_2)$ and $\w_1(z) = \w_2(z) ~\forall z \in C(\w_1)$. Then $\phi(\w_1) = \phi(\w_2)$. In other words the recoded configuration $\phi(\w)$ is completely defined by the arrow values on the set $C(\w)$.
    \item If the origin $\origin$ is in a \biinfinite{} trajectory in $\w \in \mathcal{C}$, it is also in a \biinfinite{} trajectory in $\phi(\w)$. This is because the recoding leaves the arrows at points in $C(\w)$ unchanged, and $\mathcal{C}$ is invariant under $T_\alpha$. 
\end{enumerate}
\begin{proof}[Proof of~\Thmref{thm:special trajectories have positive entropy}]
    Assume that $\mathcal{C} = \{ \w \in \W_{\alpha} \colon h(\theta_{\w}) = 0 \}$ has positive measure. Since $\theta_{\w}$ is $\mathcal{I}$ measurable, $\mathcal{C}$ is $T_{\alpha}$ invariant. {Consider the factor of $(\specialpts,T_\alpha)$ defined by the map $\psi(\w) = \{T_{\alpha}^k \w(\origin)\}_{k \in \Z}$. This a $\Z$ system with finite alphabet $\directions$ that tracks the arrows along the \biinfinite{} trajectory.  Hence it follows from \Thmref{thm:shannon macmillan non ergodic case} that for $\e > 0$, we can choose $L$ large enough such that for any interval $I \subset \Z$ longer than $L$, there exists} $\mathcal{Y} \subset \W_\alpha$ with $\Prob(\mathcal{C} \setminus \mathcal{Y}) \leq \e$ and 
    \begin{equation}
        |\pi_I(\psi(\mathcal{C} \cap \mathcal{Y}))| \leq e^{ \e |I|}.
        \label{eq:configurations on good trajectories}
    \end{equation}
    In other words, the number of arrow configurations along \biinfinite{} trajectories in $\mathcal{C}$ must grow sub-exponentially off a set of $\mathbb{P}$-measure at most $\epsilon$.

    Let $(\W',\Prob',\Z^d)$ be the factor obtained by applying {the factor map $\phi$ from the} recoding lemma (\Lemref{lem:recoding lemma}) to $\mathcal{C}$ and $(\W,\Prob,\Z^d)$. The recoding leaves arrows on $C(\w)$ invariant; and therefore if $\w \in \mathcal{C}$, $\psi(\w) = \psi(\phi(\w))$. 
    Hence, an estimate similar to \eqref{eq:configurations on good trajectories} applies, and we must have
    \[
        |\pi_I(\psi(\phi(\mathcal{C} \cap \mathcal{Y})))| \leq e^{ \e |I|}.
    \]
    Since the recoded system is factor of the original, it must have positive entropy $h > 0$. Then given $\e > 0$, for all large enough rectangles $R \subset \Z^d$, \Corref{cor:shannon macmillan for ergodic measures} gives a set $M \subset \directions^R$ such that $\Prob'(\pi_R^{-1}(M)) \geq 1 - \e$, and 
    \begin{equation}
        |M| 
        \geq e^{(h - \e)|R|}.
        \label{eq:number of configurations in box from shannon macmillan}
    \end{equation}
        We return to the original $\Z^d$ system before the recoding process. Let $R$ have side-length $L$, and consider the cube $R'$ with side-length $3L$. Consider $\w \in \W$, a point in the original space before recoding. Note that if there is at least one point $z_0 \in R$ such that $\tmap^{z_0} \w \in \mathcal{C}$, then for any $y \in R$, we must have $|y - z_0|_{\infty} \leq L$, and from \eqref{eq:definition of l function mapping the origin to the nearest point in C}, it follows that $\ell(T^y \w) \in R'$.  
        In this case, the value of $\omega(z)$ at each $z \in R'$ so that $T^z\omega \in \mathcal{C}$ determines the recoded configuration $\phi(\omega)$ inside $R$.  
       Let $\mathcal{G} = \{ \w \colon | R \cap C(\w) | = 0 \}$ be the event that there do not exist any points $z \in R$ such that $T^z \w \in \mathcal{C}$. By the ergodic theorem, $L$ may be chosen large enough so that $\Prob(\mathcal{G}) \leq \e$. 
       
       We define another ``bad'' set $\mathcal{H}$ below where the bound in \eqref{eq:configurations on good trajectories} does not apply. We then bound the possible number of arrow configurations in $R'\cap C(\w)$ on $\mathcal{G} \cup \mathcal{H}$; see \eqref{eq:sumup}. By the recoding lemma, this is enough to bound $|\pi_R(\W')|$ from above and contradict \eqref{eq:number of configurations in box from shannon macmillan}. Let $Y(\w) = \{ z \in \Z^d \colon T^z \w \in \mathcal{Y} \}$ 
    \[
        b(\w) = |\partial R' \cap \left( C(\w) \setminus Y(\w) \right) | \AND a(\w) = \left| \partial R' \cap Y (\w) \right|.
    \]
    Here, $b(\w)$ represents the number of points in $\partial R'$ on ``bad'' bi-infinite trajectories in $\mathcal{C}$ where the estimate in \eqref{eq:configurations on good trajectories} does not apply. $a(\w)$ represents the number of good points in $\partial R'$. If $\mathcal{H} = \{ \w \in \W \colon b(\w) \geq \sqrt{\e}|\partial R'|\}$,  the Markov inequality  implies that 
    \[
        \Prob(\mathcal{H}) \leq \frac{1}{\sqrt{\e} |\partial R'|} \E[ b(\w) ] = \e^{-1/2} \Prob(\mathcal{C} \setminus \mathcal{Y}) \leq \sqrt{\e}.
    \]
    Therefore, with high probability, $b(\w)$ is small. We partition the good set $\W \setminus ( \mathcal{H} \cup \mathcal{G})$ as 
    \[
        \mathcal{E}_{j,k} = \{ b(\w) = j, a(\w) = k \} \cap \left( \W \setminus (\mathcal{H} \cup \mathcal{G}) \right)
    \]
    for $j=0,\ldots,\sqrt{\e} |\partial R'|$ and $k=0,\ldots,|\partial R'|$. 
    If $z \in \partial R' \cap Y (\w)$, it follows from \eqref{eq:configurations on good trajectories} that on any trajectory of length $3L$ beginning at $z$, there are at most $e^{\e 3 L}$ possible arrow configurations. If $z \in \partial R' \cap  (C(\w) \setminus Y(\w))$, there are $|\directions|^{3L}$ arrow configurations. Then in $\phi(\mathcal{E}_{j,k})$ there are at most 
    \begin{equation}
        \label{eq:estimate of number of configurations in E j in completely pos entropy proof}
        \binom{|\partial R'|}{j,k,|\partial R'| - j - k} \left[ \exp\left( \e (3L) \right)\right]^{k} \left[ \exp\left( 3L \log |\directions| \right) \right] ^{j}
        \leq 3^{|\partial R'|} \exp \left( C_1 \sqrt{\e} L^{d}  \right)
    \end{equation} 
    distinct arrow configurations in $\directions^R$. Here, $\binom{a}{b,c,a-b-c}$ is the trinomial coefficient that accounts for the number of different ways of placing points in $\mathcal{Y},\,\mathcal{C} \setminus \mathcal{Y} \AND \mathcal{C}^c$ on $\partial R'$; $C_1$ is a constant independent of $\e$ and $L$, and we have used $k \leq |\partial R'|$ and $j \leq \sqrt{\e} |\partial R'|$. Therefore, 
    \begin{align}
        \label{eq:sumup}
        \left| \pi_R(\phi\left( \W \setminus (\mathcal{H} \cup \mathcal{G})  \right) \right|
        & = 
        \sum_{j=0}^{\sqrt{\epsilon}|\partial R'|} \sum_{k=0}^{|\partial R'|} | \phi(\mathcal{E}_{j,k}) |,
        \leq \sqrt{\e} |\partial R'|^2 2^{|\partial R'|} \exp \left( C_1 \sqrt{\e} L^{d}  \right) , \nonumber\\
        & \leq  \exp\left( C_2 \sqrt{\e} L^{d}  \right),
    \end{align}
    using \eqref{eq:estimate of number of configurations in E j in completely pos entropy proof} and $|\partial R'| = O(L^{d-1})$.
    Since the constant $C_2$ is independent of $L$ and $\e$, by choosing $\e$ small and then $L$ large enough, this contradicts \Eqref{eq:number of configurations in box from shannon macmillan}.
\end{proof}

\subsection{Discrete-time symmetric simple exclusion}
\label{sec:symmetric simple exclusion}
In this section, we construct an example with bi-infinite trajectories and positive entropy.
This demonstrates that positive entropy does not guarantee coalescence even in dimension $2$. The example is based on a discrete version of the standard symmetric simple exclusion process (SSEP). Exclusion processes were introduced by Spitzer \citep{MR268959}, where typically particles move in an infinite state-space in continuous time. In continuous time, almost surely, only a finite number of particles may step together at any given time. In discrete time, ties can occur, and we introduce an additional tie breaking variable to account for this. We first describe our construction in words and rigorously define it in the next section.

\begin{rem}
Yaguchi \citep{MR867575} constructed the first example of a discrete-time exclusion process on the state-space $\Z$ where an \emph{infinite} number of particles may step simultaneously. This was a totally asymmetric simple exclusion process (TASEP). We were not aware of Yaguchi's result until we completed the paper and we certainly could have used Yaguchi's construction to prove \Thmref{thm:positive entropy trajectories}. Moreover, \citet{MR1714990} shows that Yaguchi's exclusion process is $K$, or that it has completely positive entropy. Thus, this is an example of system with bi-infinite trajectories to which our \Thmref{thm:special trajectories have positive entropy} directly applies, and shows that the trajectories themselves must carry entropy.

Note that in discrete-time TASEP, particles do not compete for spots, and this is an additional complication we take care of. Thus, we choose to include our construction of a discrete-time symmetric simple exclusion process in this paper.
\end{rem}

SSEP is also a Markov process $\eta_t(\cdot)$ on $\configspace$. In contrast to TASEP, a particle at $z$ may decide to jump to either of its neighboring vertices with equal probability if they are not currently occupied. However, if there are particles at $z$ and $z+2$, and there is no particle at $z+1$, then the particles may \emph{compete} for the spot at $z+1$. This is the additional complication that our construction addresses. Thus, a tie-breaking mechanism is required, and this is the only novel part of our construction. %

The SSEP is an evolving configuration of particles occupying locations in $\ssepstatespace$. The discrete-time evolution of a configuration is defined in words as follows:
\begin{enumerate}
    \item At any given time $t$, a set of particles are randomly selected
        for movement using iid Bernoulli random variables: particles flip a coin to see if they get to move, and if they do, they are equally likely to move either right or left by one unit (hence
        symmetric). 
    \item However, particles can only move into a location if there is not
        already a particle there at the current time (hence simple exclusion). For example, if the space immediately to the right of a particle is occupied, but the space to the left is unoccupied, then the particle may only jump left at the current time step.
    \item If two particles compete for an empty spot, there is a 
        mechanism to break ties. We describe the tie-breaking mechanism in detail below.
\end{enumerate}
In continuous time $(t \in \R^+)$, the SSEP has a standard graphical
construction using the so-called stirring process \citep{MR776231}. In continuous time, however, particles never compete for a spot (with probability $1$), and hence this is an additional complication that we account for in the discrete setting. We could not
find a discrete-time version of this process in the literature that was appropriate for our theorems, and so we will define our own based on the above rules and a discrete version of the stirring process.  %

\subsubsection{Defining the SSEP}
Let $\configspace = \{0,1\}^{\ssepstatespace}$ be the space of particle configurations. In a configuration $\zeta \in X$, $x \in \mathbb{Z}$ has a particle if $\zeta(x) = 1$, and if $\zeta(x) = 0$, $x$ is empty. Let $\nu$ be a measure on initial configurations $\configspace$ at time $t = t_0$. The SSEP is a Markov process $\eta_t(\cdot)$ on $\configspace$. {The evolution of particle configurations $\eta_t(\cdot)$ is described by a probability measure $\Prob^{\nu}$ on $\configspace{}^{\{t_0,t_0+1,\ldots\}}$ that we define below}. Expectation with respect to this measure is denoted $\E^{\nu}$. 

The stirring process begins with \emph{stirring particles} at every point on
$\ssepstatespace \times \{t_0\}$, and time proceeds on the vertical axis. Let $H(\ssepstatespace \times \Z)$ be the set of horizontal
edges on $\ssepstatespace \times \Z$. There are $\Bernoulli(p)$ random
variables $\xi(e)$ on each $e \in H(\ssepstatespace \times \Z)$ called
\emph{firing variables}. If $e$ corresponds to the edge between $(x,t)$ and
$(x+1,t)$ and $\xi(e) = 1$, then the stirring particles at $(x,t)$ and
$(x+1,t)$ are exchanged at time $t + 1$. 
Suppose adjacent edges $m_1 = ((x,t),(x+1,t))$ and $m_2 = ((x+1,t),(x+2,t))$ fire; i.e., $\xi(m_1) = \xi(m_2) = 1$. Then, it is not clear how the stirring particles at $(x,t), (x+1,t)$ and $(x+2,t)$ are supposed to behave at the next time step. So we break this tie using iid $\Uniform[0,1]$ random variables $U(e)$ associated with the each edge $e$ in $H(\ssepstatespace \times \Z)$. Suppose $\{e_1,...,e_k\}$ is a set of adjacent horizontal edges at some time $t$ that is maximal for the property that they all fire. 
Then, the smallest edge $e = \argmin_{y=e_1,\ldots,e_k} U(y)$ 
 is chosen as the winner, and
the stirring particles associated with $e$ are exchanged. Let $S(e) = (\xi(e),U(e))$ be
the pair of \emph{stirring variables} associated with each edge. We say that an
edge \emph{stirs} if it fires and wins the tie-breaker. We say that a particle $x$ stirs if one of its adjacent edges stirs.

Given $\{ S(e) \}_{e \in H(\Z \times \Z)}$, and an initial time $t_0$ and any $x \in \Z$, let $Z_x(t_0) = x$ and inductively define for $t > t_0$ for $t \in \Z$ 
\[
    Z_x(t) = 
\begin{cases} Z_x(t-1)+1 & \text{ if the edge between }(Z_x(t-1),t-1) \text{ and} \\ 
    & (Z_x(t-1)+1,t-1) \text{ stirs}\\
Z_x(t-1)-1  & \text{ if the edge between }(Z_x(t-1),t-1) \text{ and} \\
    & (Z_x(t-1)-1,t-1) \text{ stirs}\\
Z_x(t-1) & \text{ else}.
\end{cases}
\]
In other words, $Z_x(t)$ is a simple symmetric random walk with iid increments taking values in $\{-1,0,1\}$; steps are nonzero when an edge adjacent to $Z_x(t)$ stirs. Next, we define the evolution of $\eta_t$ using $Z_x(t)$. For $y \in \mathbb{Z}$, $t \geq t_0$, and $\eta_0 \in X$, let
\begin{equation}\eta_t(x)=1\quad \text{if} ~\exists~ y \suchthat{} \eta_0(y)=1 \AND Z_y(t)=x \AND \eta_t(x) =0 \text{ otherwise}.
    \label{eq:ssep evolution definition in terms of stirring particles}
\end{equation}
 Thus, given an initial configuration $\eta_0$, we can define the SSEP using the locations of the stirring particles for any time $t > 0$. The stirring process defines the transition measure $p(x,A)$ for any cylinder $A \subset X$ and $x \in X$. So given any initial configuration $\eta_0$ at time $t_0$, the measure $\Prob^{\eta_0}$ of the Markov process on $X^{\{t_0, t_0 + 1\}}$ may be defined in a standard way (see \citep[Chapter 2]{seppalainen_unpublished_exclusion}, e.g). If $\nu$ is a measure on initial configurations, let $\Prob^{\nu} = \int \Prob^{\eta} \,d\nu(\eta)$. 
    
Equivalently, one can restate \Eqref{eq:ssep evolution definition in terms of stirring particles} in terms of the backwards paths of the stirring particles.  Namely, let $\{M_{y,t}(s)\}_{t_0 \leq s \leq t}$ be a backward path that starts at $y$ at level $\Z \times \{t\}$, and steps down to the level $\Z \times \{t_0\}$. In the backwards process, particles at $x$ and $x+1$ at time $t$ swap if the edge between $(x,t-1)$ and $(x+1,t-1)$ stirred. Then, $M_{y,t}(t) = y$ and $M_{y,t}(s) = Z_{M_{y,t}(t_0)}(s)$ for $t_0 \leq s \leq t$. This ensures that the forward and backward paths of the swapping particles are consistent, and hence $\eta_t(y) = \eta_0(M_{y,t}(0))$.

Our first result is that the Bernoulli (iid) measures on $\configspace$ are stationary and ergodic for the Markov process. 
\begin{thm}
     Let the firing variables $\xi$ in the SSEP have parameter $p < 1/2$. Then, initial particle distributions given by iid $\Bernoulli$ measures on $\{0,1\}^\ssepstatespace$ are ergodic invariant measures for the SSEP.
     \label{thm:invariant measures for discrete simple exclusion process}
\end{thm}

\begin{ques}
    Can one prove that the Bernoulli measures are the only ergodic invariant measures for the discrete SSEP when the firing parameter is $p \geq 1/2$. The restriction to  $p < 1/2$ in \Thmref{thm:invariant measures for discrete simple exclusion process} is an artifact of the coupling we chose in \Propref{prop:harmonic functions can only depend on the cardinality of the set A}. Heuristically, the firing variables simply control the rate at which the particles in the SSEP step, and this should have no effect other than to ``rescale time'' in the process.
\end{ques}

The proof of \Thmref{thm:invariant measures for discrete simple exclusion process} is standard in continuous time, and the canonical version
of the proof can be found in Chapter V of \citet{MR776231}. We provide a sketch following the expository version in Sepp\"al\"ainen's unpublished textbook \citep{seppalainen_unpublished_exclusion} for completeness. Theorem \ref{thm:invariant measures
for discrete simple exclusion process} follows from the following two lemmas. 
\begin{lem}
    Any exchangeable measure on the space of configurations $X$ is invariant for the SSEP.
    \label{lem:exchangeable measures are invariant}
\end{lem}
Exchangeable measures on $X$ are those that are invariant under permutations of finite sets of coordinates. That is, a measure $\mu$ is exchangeable if for any two sets of coordinates $\{x_1,\ldots,x_k\}$ and $\{y_1,\ldots,y_k\}$, and any $a \in \{0,1\}^k$
\[
    \mu \left\{ \eta \colon \left( \eta(x_1), \ldots \eta(x_k) \right) = a  \right\}
    = \mu \left\{ \eta \colon \left( \eta(y_1), \ldots \eta(y_k) \right) = a  \right\}.
\]
Bernoulli measures on $X$ are clearly exchangeable.
\begin{lem}
    Let the firing variables $\xi$ in the SSEP have parameter $p < 1/2$. Then, any invariant measure for the SSEP must be exchangeable.
    \label{lem:invariant measures must be exchangeable}
\end{lem}
The proof of Lemma \ref{lem:exchangeable measures are invariant} appears after we state \Propref{prop:probabilistic duality between ssep and finite ssep} and \Corref{cor:invariance equation for finite SSEP}. Lemma \ref{lem:invariant measures must be exchangeable} follows from \Propref{prop:probabilistic duality between ssep and finite ssep}, \Corref{cor:invariance equation for finite SSEP} and \Propref{prop:harmonic functions can only depend on the cardinality of the set A}.

\begin{proof}[Proof of \Thmref{thm:invariant measures for discrete simple exclusion process}]
DeFinetti's theorem~\citep{definetti_probabilismo_1931} says that the set of exchangeable
measures has iid Bernoulli measures as its extreme points, and hence Bernoulli
measures are ergodic for the SSEP. 
\end{proof}

The proofs of \Lemref{lem:exchangeable measures are invariant} and \Lemref{lem:invariant measures must be exchangeable} go via a Markov process duality between the SSEP on the space of infinite configurations $X$ and a finite SSEP on finite subsets of $\ssepstatespace$ that we now define. Two processes $Z_t \colon \Gamma \to \Gamma$ and $W_t \colon \Lambda \to \Lambda$ are said to be dual with respect to a function $G \colon \Gamma \times \Lambda \to \R$ if $\E^z \left[ G(Z_t,w) \right] = \E^w \left[ G(z,W_t) \right]$ for all $(z,w) \in \Gamma \times \Lambda$  and $t \in \mathbb{Z}^+$. 

Let $Y$ be the set of all finite subsets of $\ssepstatespace$. If $A \in Y$, let $A_t$ be the set of particles at time $0$ that ended up in $A$ at time $t$; i.e., $A_t = \{ 
M_{x,t}(0) \colon x \in A \}$. We call $A_t$ the finite SSEP. With a little abuse of notation, let $\hat{\Prob}^A$ and $\hat\E^A$ represent the probability and expectation of the finite SSEP with initial state $A$. 
\begin{prop}
    For all $t\in \mathbb{Z}^+$ and $A \in Y$
    \[
        \Prob^{\eta} \left( \eta_t(x) = 1 ~ \forall \, x \in A \right) 
        = \hat\Prob^{A} \left( \eta(x) = 1 ~ \forall \, x \in A_t \right) .
    \]
    \label{prop:probabilistic duality between ssep and finite ssep}
\end{prop}
Proposition~\ref{prop:probabilistic duality between ssep and finite ssep} implies that the SSEP on $X$ is dual to the finite SSEP on $Y$ with duality function $G(\eta,A) = \prod_{x \in A}\eta(x)$. This is because 
\begin{multline*}
    \E^{\eta}  \left[  \prod_{x \in A} \eta_t(x) \right] = \Prob^{\eta} \left( \eta_t(x) = 1 ~ \forall \, x \in A \right) \\
    = \hat\Prob^{A} \left( \eta(x) = 1 ~ \forall \, x \in A_t \right) 
    = \hat{E}^{A}\left[  \prod_{x \in A_t} \eta(x)  \right].
\end{multline*}
The full proof of \Propref{prop:probabilistic duality between ssep and finite ssep} is in \citep[Theorem 5.3]{seppalainen_unpublished_exclusion}, and since it goes through unchanged, we will not repeat it. For any probability measure $\mu$ on $X$, define the following measure on $Y$:
\[
    \hat{\mu}(A) = \mu \left\{ \eta \in X \colon \eta(x) = 1 \forall \, x \in A \right\}.
\]
Then, if $\mu$ and $\nu$ are two probability measures on $X$, $\mu = \nu$ iff
$\hat{\mu} = \hat{\nu}$.

\begin{lem}
A measure $\mu$ on $X$ is exchangeable iff {there exists $f\colon \Z^+ \to [0,1]$ such that for any finite $A \subset \Z$} we have $\hat{\mu}(A) = f(|A|)$.
\label{lem:exchangeable only if dependent on cardinality}
\end{lem}
In other words, $\hat{\mu}(A)$ depends only on the cardinality of $A$. A proof of \Lemref{lem:exchangeable only if dependent on cardinality} may be found in \cite[Appendix A.6]{seppalainen_unpublished_exclusion}. %
Let $\mu$ be a measure on $X$, and let $\mu_t$ be the measure on the configuration at time $t$.  That is, for any measurable $B \subset X$, $\mu_t = \Prob^{\mu}(\eta_t \in B)$. 
\begin{cor}[of \Propref{prop:probabilistic duality between ssep and finite ssep}]
    For any finite $A \subset \Z$
    \begin{equation*}
        \widehat{\mu_t}(A) = \hat\E^A[\hat{\mu}(A_t)].
    \end{equation*}
    \label{cor:invariance equation for finite SSEP}
\end{cor}
This is an integrated version of the duality relationship in \Propref{prop:probabilistic duality between ssep and finite ssep}. The proof of \Corref{cor:invariance equation for finite SSEP} follows \citet[Corollary 5.4]{seppalainen_unpublished_exclusion}, and again, we will not repeat it.
From~\Corref{cor:invariance equation for finite SSEP} and \Propref{prop:probabilistic duality between ssep and finite ssep}, it follows that if $\mu$ is a measure on $X$, then $\mu$ is invariant for the SSEP iff $\hat{\mu}$ is invariant for the finite SSEP. This allows us to complete the proof of \Lemref{lem:exchangeable measures are invariant} which says that exchangeable measures are invariant for the SSEP.
\begin{proof}[Proof of \Lemref{lem:exchangeable measures are invariant}]
    If ${\mu}$ is an exchangeable measure on $X$, then $\hat{\mu}(A) = f(|A|)$ for some function $f \colon \Z^+ \to [0,1]$. By construction, $|A_t|$ = $|A|$ in the finite SSEP, and from \Corref{cor:invariance equation for finite SSEP}, 
    \[
        \widehat{\mu_t}(A) = \hat{\E}^A[\hat\mu(A_t)] = \hat\E^A[f(|A_t|)] = \hat{\mu}(A).        
    \]
\end{proof}

Next, we show that if $\mu$ is an invariant measure, then $\hat{\mu}(A) = g(|A|)$ for some function $g \colon \Z^+ \to [0,1]$. Hence $\mu$ is exchangeable, and this completes the proof of \Lemref{lem:invariant measures must be exchangeable}. This is the content of the following proposition, which is the discrete analog of \citet[Prop 5.7]{seppalainen_unpublished_exclusion}. This is the \emph{only} place where our construction differs from the continuous version of the SSEP.
\begin{prop} 
    Let the firing variables $\xi$ have Bernoulli parameter $p < 1/2$. Let $f \colon Y \to \R$ be a bounded function and suppose that $f(A) = \tilde{E}^A[ f(A_t)]$ for all sets $A \in Y$ $($it is harmonic for the finite SSEP$)$, then $f(A)$ depends only on the cardinality of $A$. 
    \label{prop:harmonic functions can only depend on the cardinality of the set A}
\end{prop}
The proposition applies in particular to any invariant measure $f(A) = \hat{\mu}(A)$ on $Y$. 
\begin{proof}
    {For $A$ and $B$ that are finite subsets of $\Z$, we define a coupling between $A_t$ and $B_t$ such that they evolve according to the rules of finite SSEP on the same space.} As in~\citep[Proposition 5.7]{seppalainen_unpublished_exclusion}, it suffices to show that if $A$ and $B$ are any two subset of $\Z$ cardinality $n$ that have $n-1$ points in common, then $f(A) = f(B)$. This is because any set can be transformed into any other by changing one point at a time. Let $A_t = C_t \cup \alpha_t$ and $B_t = C_t \cup \beta_t$, where $C_t = A_t \cap B_t$ is the common set with $n-1$ points. The proof proceeds by constructing a successful coupling between the two processes; i.e., by constructing $A_t$ and $B_t$ on the same space such that $A_t = B_t$ eventually almost surely. Then,
    \begin{equation}
        \begin{aligned}
            |f(A) - f(B)| & = | \hat\E^A f(A_t) - \hat\E^B f(B_t) | \\
            & \leq \hat\E |f(A_t) - \E f(B_t) | \\
            & \leq 2 \Norm{f}{\infty} \hat\Prob(A_t \neq B_t) \to 0.
        \end{aligned}
        \label{eq:successful coupling implies harmonic functions depend only on cardinality}
    \end{equation}
    The coupling lets $\alpha_t$ and $\beta_t$ alone evolve independently until $\alpha_t = \beta_t$, after which they evolve together. 

    Since the finite SSEP is a time-homogeneous countable-state Markov chain, it is enough to define the one-step transition probabilities of the triple $(C,\alpha,\beta) \in Y \times \Z \times \Z$ at some fixed time $t$; hence we will drop the $t$ subscript in the following. We first give a heuristic description of the coupling between $\alpha$ and $\beta$, and the formal definition will be given afterwards. The individual particles in the three sets $C$, $\alpha$ and $\beta$ evolve using the stirring particles. However, using a new tie-breaking rule, $\alpha$ and $\beta$ will not be allowed to step at the same time $t$ when they are on opposite ends of the same edge; but $\alpha$ or $\beta$ may jump onto the other at this time. Once $\alpha$ and $\beta$ occupy the same position in $\Z$, they evolve together. Neither $\alpha$ nor $\beta$ are allowed to overlap with $C$. 
    
    Let $\left\{ \xi^*(e) \colon e \to \{L,0,R\} \right\}_{e \in H(\Z \times \Z)}$ be an iid family independent of the stirring variables and the initial probability measure, such that 
    $\Prob(\xi^*(e) = L) = \Prob(\xi^*(e) = R) = p$. Suppose $\alpha$ and $\beta$ and share an edge $e$ and wlog, assume $\alpha < \beta$. If $\xi^* = L$, then $\alpha$ is chosen to fire, and we go through the tie-breaking stage as usual and determine if the stirring particle associated with $\alpha$ steps. The transition probabilities $(C,\alpha,\beta) \to (C',\alpha',\beta')$ are constructed using the firing variables $\xi$ except when $\alpha$ and $\beta$ share an edge $e$, where we use $\xi^*$ instead. The same is done with $\beta$ if $\xi^* = R$ instead. If $\xi^*(e) = 0$, then the edge does not fire. Note that marginal transition probabilities $(C,\alpha) \to (C',\alpha')$ and $(C,\beta) \to (C',\beta')$ match those of the stirring particles in the original finite SSEP. This is because $\xi^*(e)$ behaves like $\Bernoulli(p)$ when deciding whether or not $e$ fires for $\alpha$ (or $\beta$). If $\alpha = \beta$ at some time $t$, $C \cup \alpha$ simply evolves using the standard rules of the finite SSEP in the future.

    Next, we show that under the above coupling, $\Prob(\alpha_t = \beta_t \text{ eventually}) = 1$. By \eqref{eq:successful coupling implies harmonic functions depend only on cardinality}, this completes the proof. The process $Z_t = \alpha_t - \beta_t$ is a simple random walk with independent increments taking values in $\{0,\pm 1,\pm 2\}$ until $|\alpha_t - \beta_t | = 1$. 
    To see this, by horizontal translation invariance of the measure on the firing variables, it is equally likely that the edge to the left of $\alpha_t$ (resp $\beta_t$) and the edge to the right of $\alpha_t$ (resp $\beta_t$) stir; i.e., it is equally likely to go left or right. This shows that $Z_t$ is a random walk with mean-zero increments until $|\alpha_t - \beta_t | = 1$. It is a well-known fact that such a one-dimensional random walk recurs infinitely often to $|Z_t| = 1$ with probability $1$. Every time $|Z_t| = 1$, there is a positive probability that $\alpha_t$ and $\beta_t$ merge. Furthermore at different times $t_1$ and $t_2$ when $|Z_{t_1}| = |Z_{t_2}| = 1$, the events that they merge are independent. Therefore $Z_t$ recurs to $0$ eventually almost surely.
\end{proof}

To cast the SSEP in our framework, we build a measure $\tilde\Prob$ on an extended space $\tilde\Omega = \left( \{0,1\} \times \{0,1\} \times [0,1] \right)^{\ssepstatespace \times \Z}$ that is invariant and ergodic under $\Z^2$ shifts. We need to keep track of the firing and tie-breaking variables to be able to follow individual particle trajectories, since tracking just the evolution of the configuration $\eta_t(x)$ is insufficient. Each $\w \in \tilde\W$ is a realization of the configuration, stirring and tie-breaking variables; i.e., $\w(x,t)  = (\eta_t(x), \xi(x,t), U(x,t))$ where the stirring variables are associated with the edge between $(x,t) \AND (x+1,t)$. For the interval of times $I_n = \{-n,-n+1,\ldots,\infty\}$, let $\tilde\Prob_n^{\nu}$ be the measure on $\tilde\Omega_n = \left( \{0,1\} \times \{0,1\} \times [0,1] \right)^{I_n \times \Z}$ obtained by starting the SSEP at $t = -n$ using a $\Bernoulli(p)$ product measure $\nu$ in the half-plane $\Z \times I_n $. The following is a consequence of Kolmogorov's consistency theorem, which can be found in \citep[Theorem 1.1]{seppalainen_unpublished_exclusion} or \citep[Chapter 12]{MR1068527}.

\begin{prop}
   There exists a measure $\tilde\Prob$ on $\tilde\Omega$ that projects consistently onto $(\tilde\Omega_n,\tilde\Prob_n^\nu)$ in the sense of the Kolmogorov consistency theorem. Further, $\tilde\Prob$ is invariant and ergodic under horizontal and vertical shifts of the lattice. 
   \label{prop:construction of stationary measure on ssep using kolmogorov consistency theorem}
\end{prop}
\begin{proof}
   $\tilde{\Prob}_n^\nu$ forms a consistent family of probability measures on $\tilde{\Omega}_n$ in the sense of Kolmogorov's consistency theorem. See \citet[Theorem 1.1]{seppalainen_unpublished_exclusion} for the definition of consistency. {We think of time running on the vertical axis.} $\tilde{\Prob}_n^{\nu}$ is ergodic under the vertical shift since $\nu$ is an ergodic measure for the process. $\tilde{\Prob}_n^\nu$ is invariant under horizontal shifts since $\nu$ is Bernoulli and the stirring variables are iid. 
\end{proof}

The family of stationary compatible walks correspond to particle trajectories in the exclusion process. Suppose the particle at $(x,t)$ goes to $(x+1,t+1)$. Then there is an arrow $\alpha$ (see \eqref{eq:direction function definition}) connecting $(x,t)$ and $(x+1,t+1)$. These are not nearest-neighbor walks, but our framework is trivially extended to this setting: one way is to replace the arrow going from $(x,t)$ to $(x+1,t+1)$ by two arrows going from $(x,t)$ to $(x,t+1)$ and then from $(x,t+1)$ to $(x+1,t+1)$. The arrow map $\alpha$ describing the particle trajectories is some complicated, but fairly explicit function on $\tilde\W$, that we do not spell out.

Now that we have cast the SSEP in our setting, we prove that it has positive entropy:
\begin{thm}
    The stationary SSEP system $(\tilde{\W},\tilde{\Prob},\Z^2)$ described in \Propref{prop:construction of stationary measure on ssep using kolmogorov consistency theorem} has positive entropy.
    \label{thm:the ssep system has positive entropy}
\end{thm}
\color{black}
\begin{proof}%
    Let ${\Omega}=\{0,1\}^{\mathbb{Z}\times \mathbb{Z}}$. Let $\pi:\tilde{\Omega} \to {\Omega}$ be projection onto the first coordinate, and let ${\mathbb{P}}=\pi_* \tilde{\mathbb{P}}$ be the pushforward measure. It is enough to show that $({\Omega},{\mathbb{P}},\mathbb{Z}^2)$ has positive entropy to prove \Thmref{thm:the ssep system has positive entropy} since entropy can only decrease under the factor map $\pi: \tilde\Omega \to \Omega$. Thus, \Thmref{thm:the ssep system has positive entropy} follows from \Propref{prop:pos entropy} below.
\end{proof}
\begin{prop}\label{prop:pos entropy} $({\Omega},{\mathbb{P}},\mathbb{Z}^2)$ has positive entropy.
\end{prop}
Lemmas \ref{lem:bernoulli} and \ref{lem:it suffices to show conditional entropy of ssep is pretty large} prove that $({\Omega},{\mathbb{P}},\mathbb{Z}^2)$ has positive entropy. 
\begin{lem}\label{lem:bernoulli} Restricted to each horizontal line $\mathbb{P}$ is a Bernoulli measure. 
\end{lem}
\begin{proof}
    Fix some horizontal line $t = -n$. $\Prob_n^\nu$ is a Bernoulli measure $\nu$ on this horizontal line, and hence so is $\Prob$ by consistency and the vertical shift invariance of  $\Prob$.
\end{proof}

Let $\mathcal{B}_{m,n}=\{0,1\}^{m \times n}$ be the set of particle configurations in a $m \times n$ rectangular subset of $\Z^2$, and let $\pi_{m,n}$ be the projection from $\tilde{\W}$ to $\mathcal{B}_{m,n}$. For $b \in\mathcal{B}_{m,n}$, we will use the shorthand $\Prob(b)$ for $\Prob \circ \pi_{m,n}^{-1}(b)$ in the following. 

The conditional entropy is defined as usual for $i = 1,2,\ldots$ by
\[
    H(\mathcal{B}_{m,i+1} |\mathcal{B}_{m,i}) = \sum_{a \in \mathcal{B}_{m,i}, b \in \mathcal{B}_{m,i+1} }  \Prob(b|a) \log \Prob(b | a) \Prob(a) .
\] 
Lemma \ref{lem:it suffices to show conditional entropy of ssep is pretty large} shows that it is enough to get a lower bound on the conditional entropy of the SSEP to prove \Propref{prop:pos entropy}.  
\begin{lem}
\label{lem:it suffices to show conditional entropy of ssep is pretty large}
    To prove \Propref{prop:pos entropy}, it suffices to show that there exists $c>0$ so that for all large enough $n,m$ we have
    \[
        H(\mathcal{B}_{m,n+1} |\mathcal{B}_{m,n}) \geq cm.
    \]
\end{lem}

\begin{proof}
    A standard calculation shows that the entropy can be decomposed as a sum of conditional entropies as follows \citep[Ch.\,9.2, Prop 2.1]{MR797411}:
    \[
         H(\mathcal{B}_{m,n})
        = H(\mathcal{B}_{m,1}) + \sum_{i=1}^{n-1} H(\mathcal{B}_{m,i+1} |\mathcal{B}_{m,i}).
    \]
    Therefore, the entropy rate is
    \begin{equation}
    h(\Prob) 
    = {\lim_{m,n \to \infty}} \, \frac 1 m H(\mathcal{B}_{m,n+1}|\mathcal{B}_{m,n}) \geq c.
    \label{eq:entropy rate is positive if conditional entropies are positive}
    \end{equation}
\end{proof}

\begin{lem}
   There exists $c>0$ so that  for all $n > 0$ and $m \geq 7$,
    \[
        H(\mathcal{B}_{m,n+1} |\mathcal{B}_{m,n}) > cm.
    \]
    \label{lem:conditional entropy of ssep is large}
\end{lem}

\begin{proof} 
    Given $a \in \mathcal{B}_{m,n}$, we say a spot on the top row $(x,n)$ is \emph{free} if $a_{x,n}=1$ and $a_{x\pm i,n} = 0$ for $i=1,2,3$.  
    We first show that there exists $q>0$ so that if $a\in \mathcal{B}_{m,n}$ and the top row has $r$ free entries then 
    \begin{equation}
        -\frac 1 {\Prob(a)}\sum_{b' \in \mathcal{B}_{m,n+1}}\Prob(b'\cap a)\log\left(\frac{\Prob(b'\cap a)}{\Prob(a)} \right)\geq rq.
        \label{eq:conditional entropy in ssep is large}
    \end{equation} 

     We now consider a partition $\mathcal{P}'$ of $\mathcal{B}_{m,n+1}$ by refining $\mathcal{B}_{m,n}$ according to the previously described outcomes at the free entries. Let $a \in \mathcal{B}_{m,n}$, and suppose it has $r \geq 0$ free sites. There are $3^r$ outcomes possible at the free sites {in the following sense: each particle can stay put, move left or move right}. For each of these outcomes on the free sites, define a set $B_{\beta,a}$ ($\beta=1,\ldots,3^r$) that consists of the points in $\mathcal{B}_{m,n+1}$ that match $a$ on the first $n$ rows, and the outcome on the $n+1$\textsuperscript{th} row at the free sites. Then $\mathcal{P}'$ is the partition formed by taking a union over these sets and the different points $a \in \mathcal{B}_{m,n}$.
     
     Let $b \in \mathcal{B}_{m,n+1}$. Let $p_k$ be the conditional probability given $a$ that exactly $k$ specified particles stir in the $r$ free spots and the other particles in the free spots stay put. To get a particle to stir, it is sufficient if exactly one of the adjacent edges fires; to force it to not stir, it is sufficient if both its adjacent edges do not fire. Since these $r$ particles are free, the edges involved in these events for different particles are disjoint ---particles are separated by at least three empty spots--- and consequently the associated firing variables are independent. Hence,
     \begin{equation}
         p_k \geq (2p(1-p)^2)^k ((1-p)^2)^{r-k}.
         \label{eq:lower bound for free particle stirring}
     \end{equation}
     {To explain this bound, suppose we have the following particle configuration
         \begin{equation} 
         \circ \overset{a}- \circ \overset{b}- \bullet \overset{c}- \circ \overset{d}- \circ - \circ - \bullet -
         \label{eq:picture for particle movement in ssep}
     \end{equation}
         where filled circles represent particles, the open circles represent empty sites, and the dashes represent edges. If the first particle moves, then one way this can happen is if the edge labeled $b$ fires, and $a$ and $c$ do not fire; a second way is if $b$ and $d$ do not fire, but $c$ fires. This event has probability $2p(1-p)^2$ and appears in the first term of the RHS of \eqref{eq:lower bound for free particle stirring}. If the first particle does not move, then one way this could happen is if both edges $b$ and $c$ do not fire; the probability of this event is $(1-p)^2$ and this appears in the second term of \eqref{eq:lower bound for free particle stirring}. Note that by the definition of free, the edges in these events for different free particles are disjoint. 
     }
     
     To upper bound $p_k$, note that the probability that $k$ specified free particles stir is bounded above by the probability that at least one of the adjacent edges of each these $k$ particles \emph{must} fire: 
     \begin{equation}
         p_k \leq (1 - (1-p)^2)^k = (p(2-p))^k.
         \label{eq:upper bound for free particle stirring}
     \end{equation}
     Then, using \eqref{eq:upper bound for free particle stirring}, \eqref{eq:lower bound for free particle stirring} and the expectation of the Binomial distribution, we get
\begin{align}\label{eq:free est} 
    -\frac 1 {\Prob(a)}\sum_{B' \in \mathcal{P}'} & \Prob(a\cap B')\log
    \left( \frac{\Prob(B'\cap a)}{\Prob(a)} \right) \nonumber \\
    & = - \sum_{k=0}^r \binom{r}{k} p_k \log\left( p_k \right), \nonumber \\
    & \geq r c,
\end{align} 
for some $p$ dependent constant $c$. Notice that because $\mathcal{P}'$ is a coarser partition than $\mathcal{B}_{n,m+1}$ and entropy can only increase on refinement, 
this is a lower bound:
$$-\frac 1 {\Prob(a)}\sum_{b \in \mathcal{B}_{m,n+1}}\Prob(a\cap b)\log(\Prob(b\cap a)) \geq rc. $$  
Taking expectation over $a \in \mathcal{B}_{m,n}$, we get
\[
     H(\mathcal{B}_{m,n+1} |\mathcal{B}_{m,n}) \geq c \, \E[ \# \text{ of free sites on the top row of an } n \times m \text{ box}] = c q (m-6),
\]
where $q$ is the probability that a particle is free, and the $m-6$ accounts for edge effects. 
\end{proof}

Let $(\directions^{\Z^2},\Prob,\{T^z\}_{z\in \mathbb{Z}_2})$ in the statement of \Thmref{thm:positive entropy trajectories} be the stationary SSEP system $(\W,\Prob,\Z^2)$ from \Propref{prop:pos entropy}, and let $(\specialpts,\Prob_{\alpha},T_{\alpha})$ be the $\mathbb{Z}$ system we obtain on \biinfinite{} particle trajectories in the SSEP by following the prescription in the paragraph after \eqref{eq:translation map along walks}. \Thmref{thm:positive entropy trajectories} states that $(\specialpts,\Prob_{\alpha},T_{\alpha})$ has positive entropy. 
\begin{proof}[Proof of {\Thmref{thm:positive entropy trajectories}}]%
    The proof follows from the contrapositive: suppose the entropy $(\specialpts,\Prob_{\alpha},T_{\alpha})$ is $0$. Then, for any partition $\mathcal{P}$ of $\specialpts{}$, we must have $h(\Prob_{\alpha},\mathcal{P}) = 0$. We show that this contradicts \Lemref{lem:it suffices to show conditional entropy of ssep is pretty large} and \Lemref{lem:conditional entropy of ssep is large}, and therefore the entropy must be positive. Let $\mathcal{P}$ be the partition obtained by splitting $\specialpts{}$ based on the three values of the arrows $\alpha \in \{(0,1),(1,1),(-1,1)\}$. 
    Given $\e > 0$, for large enough $m$, we must have that all but $\epsilon$ of the measure of $\Prob_\alpha$ is {supported on} $(1+\epsilon)^m$ paths of length $m$. We call these paths \emph{predictable}.%

    We now imitate the proof of Theorem \ref{thm:special trajectories have positive entropy}. We now consider $\hat{\mathcal{P}}_{n,m}$, the partition of $\Omega$ given by the $2^{nm}$ elements of $\mathcal{B}_{n,m}$.
     Any path that enters a rectangle $R$ of size $n \times m$, must be in the rectangle $R'$ of size $(n+2m) \times m$ surrounding it. Consider the bottom row $R'$; each spot on the bottom row can either contain a particle or not. Fixing the places in the bottom row that have a particle, the configuration  of $R$ is determined by the paths of the particles in the bottom row.
     We now count the number of elements of $\hat{\mathcal{P}}_{n,m}$  that have at most $L\epsilon$ proportion of the paths {on the bottom row} that are not predictable: %
     \begin{multline} 
     \sum_{ k=0 }^{n+2m}\sum_{j=0}^{\lfloor L\epsilon k -\rfloor }\binom{n+2m}{j,k-j,n+2m-k}((1+\epsilon)^m)^{k-j}(3^m)^j \\
     \leq (n+2m)2^{n+2m}(1+\epsilon)^{n+2m}( 3)^{L\epsilon (n+2m)},
     \label{eq:number of paths in good set in ssep entropy calc}
     \end{multline}
     where $k$ represents the number of places on the bottom row of the rectangle that have particles and $j$ is the number of particles that have non-predictable paths. We now want to dominate $H(\hat{\mathcal{P}}_{n,m})$. The contribution to $H(\hat{\mathcal{P}}_{n,m})$ is maximized by assuming that each of the pieces in \eqref{eq:number of paths in good set in ssep entropy calc} has equal weight, and thus they contribute 
     $$
     \log(n+2m)+(n+2m)\log(2)+(n+2m)\log(1+\epsilon)+L\epsilon(n+2m)\log( 3^m)
     $$
     to $H(\hat{\mathcal{P}}_{n,m})$.
     Using the Markov inequality, the measure of the union of elements in $\hat{\mathcal{P}}_{n,m}$ with more than an $L\epsilon$ proportion of their paths not predictable is at most $L^{-1}$. So we trivially dominate 
     \begin{multline*}
         H(\hat{\mathcal{P}}_{n,m})\leq \log(n+2m)+(n+2m)\log(2)+(n+2m)\log(1+\epsilon) \\ +L\epsilon(n+2m)\log( 3^m) +\frac 1 L (n + 2m)m\log(3). 
     \end{multline*}
     If $L=\frac 1{\sqrt{\epsilon}}$ we have 
     $$\underset{m\to \infty}{\limsup}\, \frac 1 {m^2} H(\hat{\mathcal{P}}_{m,m})\leq 3\sqrt{\epsilon}\log(3)+{3}\sqrt{\epsilon}\log(3).$$
     This contradicts the fact that $({\Omega},{\mathbb{P}},\mathbb{Z}^2)$ has positive entropy and completes the proof.
\end{proof}

\printbibliography

\end{document}